\documentclass[11pt]{amsart}
\usepackage{amsmath}
\usepackage{amssymb}
\usepackage{amsthm}
\usepackage{stackengine}
\usepackage{latexsym}
\usepackage{empheq}

\def\NZQ{\mathbb}               % the font for N,Z,Q,R,C

\def\F2{{\NZQ F}_2}

\def\opn#1#2{\def#1{\operatorname{#2}}} % to make operators
\opn\chara{char} \opn\length{\ell} \opn\pd{pd} \opn\rk{rk}
\opn\projdim{proj\,dim} \opn\injdim{inj\,dim} \opn\rank{rank}
\opn\depth{depth} \opn\codepth{codepth} \opn\grade{grade}
\opn\height{height} \opn\embdim{emb\,dim} \opn\codim{codim}

\opn\Tr{Tr} \opn\bigrank{big\,rank}
\opn\superheight{superheight}\opn\lcm{lcm}
\opn\trdeg{tr\,deg}%
\opn\reg{reg} \opn\lreg{lreg} \opn\skel{skel}
\opn\Gr{Gr}
\opn\ann{ann}
\opn\sign{sign}

\opn\div{div} \opn\Div{Div} \opn\cl{cl} \opn\Cl{Cl}
\opn\Spec{Spec} \opn\Supp{Supp} \opn\supp{supp} \opn\Sing{Sing}
\opn\Ass{Ass}\opn\fdepth{fdepth}
\opn\Ann{Ann} \opn\Rad{Rad} \opn\Soc{Soc}
\opn\Sym{Sym} \opn\Ker{Ker} \opn\Coker{Coker} \opn\Im{Im}
\opn\Hom{Hom} \opn\Tor{Tor} \opn\Ext{Ext} \opn\End{End}
\opn\Aut{Aut} \opn\id{id} \opn\ini{in} \opn\tr{tr}

\opn\nat{nat}\opn\it{it}
\opn\pff{proof}%   \pf exists already
\opn\Pf{proof} \opn\GL{GL} \opn\SL{SL} \opn\mod{mod} \opn\ord{ord}
\opn\aff{aff} \opn\con{conv} \opn\relint{relint} \opn\st{st}
\opn\lk{lk} \opn\cn{cn} \opn\core{core} \opn\vol{vol}
\opn\link{link} \opn\star{star} \opn\skel{skel} \opn\indeg{indeg}
\opn\Ass{Ass} \opn\Min{Min} \opn\sdepth{sdepth} \opn\depth{depth}
\opn\gr{gr}

\def\pot#1#2{#1[\kern-0.28ex[#2]\kern-0.28ex]}

\opn\dirlim{\underrightarrow{\lim}}
\opn\inivlim{\underleftarrow{\lim}}
\def\Implies{\ifmmode\Longrightarrow \else
     \unskip${}\Longrightarrow{}$\ignorespaces\fi}
\def\implies{\ifmmode\Rightarrow \else
     \unskip${}\Rightarrow{}$\ignorespaces\fi}
\def\iff{\ifmmode\Longleftrightarrow \else
     \unskip${}\Longleftrightarrow{}$\ignorespaces\fi}

\let\:=\colon
\opn\d{d}
\newtheorem{Theorem}{Theorem}[section]
\newtheorem{Lemma}[Theorem]{Lemma}
\newtheorem{Corollary}[Theorem]{Corollary}

\newtheorem{Example}[Theorem]{Example}

\newtheorem{Definition}[Theorem]{Definition}

\let\epsilon\varepsilon
\let\phi=\varphi
\let\kappa=\varkappa
\def\qed{\ifhmode\textqed\fi
   \ifmmode\ifinner\quad\qedsymbol\else\dispqed\fi\fi}
\def\textqed{\unskip\nobreak\penalty50
    \hskip2em\hbox{}\nobreak\hfil\qedsymbol
    \parfillskip=0pt \finalhyphendemerits=0}
\def\dispqed{\rlap{\qquad\qedsymbol}}

\opn\Gin{Gin}

\opn\inii{in} \opn\inim{inm} \opn\rate{rate}

\numberwithin{equation}{section}

\title{Chordal circulant graphs and induced matching number}

\usepackage{graphicx}
\usepackage{xcolor}
\usepackage{tikz}

\address{Department of Mathematics\\
University of Trento\\
via Sommarive, 14\\
38123 Povo (Trento), Italy
}
\author{Francesco Romeo}
\date{}

\begin{document}
\maketitle
\begin{abstract}
Let $G=C_{n}(S)$ be a circulant graph on $n$ vertices. In this paper we characterize chordal circulant graphs and then we compute $\nu (G)$, the induced matching number of $G$. These latter are useful in bounding the Castelnuovo-Mumford regularity of the edge ring of $G$.
\end{abstract}

\section*{Introduction}\label{sec:intro}
Let $G$ be a finite simple graph with vertex set $V(G)$ and edge set $E(G)$. 
Let $\mathcal{C}$ be a cycle of $G$. An edge $\{v,w\}$ in $E(G) \setminus E(\mathcal{C})$ with $v,w$ in $V(\mathcal{C})$ is a \textit{chord} of $\mathcal{C}$.
A graph $G$ is said to be \textit{chordal} if every cycle has a chord. \\
We recall that a circulant graph is defined as follows. Let $S\subseteq T:= \{ 1,2,\ldots,\left \lfloor\frac{n}{2}\right \rfloor\}$. The \textit{circulant graph} $G:=C_n(S)$ is a simple graph with $V(G)=\mathbb{Z}_n=\{0,\ldots,n-1\}$ and $E(G) := \{ \{i, j\} \mid |j-i|_n \in S \}$ where $|k|_n=\min\{|k|,n-|k|\}$. Given $i,j \in V(G)$ we call \textit{labelling distance} the number $|i-j|_n$. By abuse of notation we write $C_n(a_1,a_2,\ldots,a_s)$ instead of $C_n(\{a_1,a_2,\ldots, a_s\})$. \\ 
Circulant graphs have been studied under combinatorial (\cite{BH0,BH1}) and algebraic (\cite{Ri}) points of view. In the former, the authors studied some families of circulants, i.e. the $d$-th powers of a cycle, namely the circulants $C_n(1,2,\ldots,d)$ (that we will analyse in Section \ref{sec:two}) and their complements . In the latter, the author studied some properties of the edge ideal of circulants. Let $R= K[x_0, \dots, x_{n-1}]$ be the polynomial ring on $n$ variables over a field $K$. The \textit{edge ideal} of $G$, denoted by $I(G)$, is the ideal of $R$ generated by all square-free monomials $x_i x_j$ such that $\{i,j\} \in E(G)$. The quotient ring $R/I(G)$ is called \textit{edge ring} of $G$. Some algebraic properties and invariants of $R/I(G)$ can be derived from combinatorial properties of $G$. Chordality and the induced matching number have been used to give bounds on the Castelnuovo-Mumford regularity of $R/I(G)$ (see Section \ref{sec:pre}). 

In Section \ref{sec:one} we prove that a circulant graph is chordal if and only if it is either complete or a disjoint union of complete graphs.\\
In Section \ref{sec:two} we give an explicit formula for the induced matching number of a circulant graph $C_n(S)$ depending on the cardinality and the structure of the set $S$. Moreover, by using \texttt{Macaulay2}, we compare the Castelnuovo-Mumford regularity of $R/I(G)$ with $\nu(G)$, the lower bound of Theorem \ref{kat}, when $G$ is the $d$-th power of a cycle and $n$ is less than or equal to 15. We report the result in Table \ref{Tab}.

\section{Preliminaries}\label{sec:pre}
In this section we recall some concepts and notation that we will use later on in this article.\\

We recall that the circulant graph $C_n(1,2, \ldots , \lfloor \frac{n}{2} \rfloor )$ is the complete graph $K_n$.
Moreover, we compute the number of components of a circulant graph with the following 
\begin{Lemma}\label{comp}
Let $S=\lbrace a_{1},\ldots,a_{r}\rbrace$ be a subset of $T$ and let $G=C_n(S)$ be a circulant graph. Then $G$ has $\gcd(n,a_{1},\ldots,a_{r})$ disjoint components. In particular, $G$ is connected if and only if $\gcd(n,a_{1},\ldots,a_{r})=1$.
\end{Lemma}
For a proof see \cite{BT}.
From Lemma \ref{comp} it follows that if $n=dk$, then the disjoint components of $C_n(a_1 d, a_2 d, \ldots , a_s d)$ are $d$ copies of the circulant graph $C_k(a_1 , a_2 , \ldots , a_s) $. 

Let $G$ be a graph. A collection $C$ of edges in $G$ is called an \emph{induced matching} of $G$ if the edges of $C$ are pairwise disjoint and the graph having $C$ has edge set is an induced subgraph of $G$. The maximum size of an induced matching of $G$ is called \emph{induced matching number} of $G$ and we denote it by $\nu(G)$.

Let $\mathbb{F}$ be the minimal free resolution of $R/I(G)$. Then 
\[
\mathbb{F} \ : \ 0 \rightarrow F_{p} \rightarrow F_{p-1} \rightarrow \ldots \rightarrow F_{0} \rightarrow R/I(G)\rightarrow 0
\]
where $F_{i}=\bigoplus\limits_j R(-j)^{\beta_{i,j}}$. The $\beta_{i,j}$ are called the \emph{Betti numbers} of $\mathbb{F}$.
%For any $i$, $\beta_{i}=\sum_{j} \beta_{i,j}$ is called the $i$-th \emph{total Betti number}.
The \emph{Castelnuovo-Mumford regularity} of $R/I(G)$, denoted by $\mbox{reg}\  R/I(G)$ is defined as 
\[
\mbox{reg} \ R/I(G)=\max\{j-i: \beta_{i,j} \}.
\]

Let $G$ be a graph. The \textit{complement graph} $\bar{G}$ of $G$ is the graph whose vertex set is $V(G)$ and whose edges are the non-edges of $G$. We conclude the section by stating some known results relating chordality and induced matching number to the Castelnuovo-Mumford regularity. The first one is due to Fr\"oberg (\cite[Theorem 1]{Fr})
\begin{Theorem}\label{fr}
Let $G$ be a graph. Then $\reg R/I(G) \leq 1$ if and only if $\bar{G}$ is chordal. 
\end{Theorem}
The second one is due to Katzman  (\cite[Lemma 2.2]{Ka}).
\begin{Theorem}\label{kat}
For any graph $G$, we have $\reg R/I(G)\geq \nu(G)$. 
\end{Theorem}
When $G$ is the circulant graph $C_{n}({1})$, namely the cycle on $n$ vertices, we have the following result due to Jacques (\cite{Ja}).
\begin{Theorem}\label{cyc}
Let $C_n$ be the $n$-cycle and let $I=I(C_n)$ be its edge ideal. Let $\nu=\lfloor \frac{n}{3}\rfloor$ denote the induced matching of $C_n$. Then

\[
\reg R/I=\begin{cases}
\begin{aligned}
 \   &\nu \ &\mbox{if  } \  &n &\equiv \ & 0,1  &\pmod 3& \\
  \ &\nu +1 \ &\mbox{if  } \  &n &\equiv \ &2  &\pmod 3&. \\
\end{aligned}
\end{cases}
\]
\end{Theorem}
\section{Chordality of circulants}\label{sec:one}
\label{mainsection}
The aim of this section is to prove the following 

\begin{Theorem}\label{chord}
Let $G$ be a circulant graph.
Then $G$ is chordal if and only there exists $d\geq 1$ such that $n=dm$ and $G=C_{n}(d,2d,\ldots,\lfloor\frac{m}{2}\rfloor d)$.
\end{Theorem}

The $\Leftarrow )$ implication is trivial. If $d=1$, $G$ is the complete graph $K_n$, while if $d>1$, then $G$ is the disjoint union of $d$ complete graphs $K_m$.

To prove $\Rightarrow)$ implication we need some preliminary results.

\begin{Lemma}\label{orda}
Let $G=C_n(S)$ be a circulant graph.
Let us assume that there exists $a\in S$ with $k=\ord(a)\geq 4$ such that 
\[
\Big\{a,2a,\ldots , \Big\lfloor \frac{k}{2} \Big\rfloor a \Big \} \nsubseteq S.
\]
Then $G$ is not chordal.
\end{Lemma}
\begin{proof}
Since $k\geq 4$, then $\{a\}\subset \{a,2a,\ldots , \lfloor \frac{k}{2} \rfloor a\}$.
If $ \Big\{a,2a,\ldots, \Big\lfloor \frac{k}{2} \Big\rfloor a \Big \} \nsubseteq S$ then we have two cases:
\begin{enumerate}
\item[(1$S$)] $\{a,2a,\ldots ,ra,(r+t)a \} \subseteq S$ and $(r+1)a,\ldots, (r+t-1)a \notin S$, with $r\geq 1$ and $t \geq 2$;
\item[(2$S$)] $\{a,2a,\ldots ,ra\}\subseteq S$ and $(r+1)a,\ldots, \lfloor \frac{k}{2} \rfloor a \notin S$, with $1 \leq r < \lfloor \frac{k}{2} \rfloor$.\\
\end{enumerate}
\begin{itemize}
\item[(1$S$)] 
We want to find a non-chordal cycle of $G$. We consider the edges $\{0,(r+t)a\}$, $\{0,a\}$, $\{a,(r+1)a\}$ (see Figure \ref{nona}). If $(r+1)a$ is adjacent to $(r+t)a$, then we found a non-chordal cycle of $G$. 
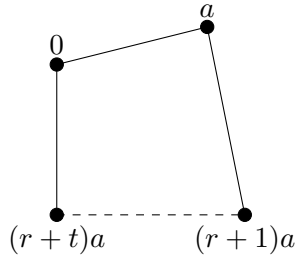
\begin{figure}[h]
\centering
\begin{tikzpicture}
\draw (-1,1) -- (1,1.5);
\draw[dashed] (-1,-1) -- (1.5,-1);
\draw (1,1.5) -- (1.5,-1);
\draw (-1,1) -- (-1,-1);
\filldraw (-1,1) circle (2.5 pt) node [anchor=south] {$0$};
\filldraw (1,1.5) circle (2.5 pt) node [anchor=south] {$a$};
\filldraw (1.5,-1) circle (2.5 pt) node [anchor=north] {$(r+1)a$};
\filldraw (-1,-1) circle (2.5 pt) node [anchor=north] {$(r+t)a$};
\end{tikzpicture}\caption{Some edges of a non-chordal cycle of ${G}$.}\label{nona}
\end{figure}
Otherwise, we apply the division algorithm to $r+t$ and $r+1$, that is
\[
r+t=(r+1)q+s \ \ \ 0 \leq s \leq r.
\]
From the vertex $(r+1)a$ we alternately add $a$ and $ra$ to get the multiples of $(r+1)a$, until $q(r+1)a$. If $s=0$, then we get $(r+t)a$, otherwise $0< s \leq r$ and $sa \in S$ so we join $q(r+1)a$ and $(r+t)a$. The above cycle has length greater than or equal to 4 because the vertices $0,a,(r+1)a,(r+t)a$ are different. Furthermore, it is non-chordal because by construction any pair of non-adjacent vertices in the cycle has labelling distance in $\{(r+1)a,\ldots, (r+t-1)a\}$. 
\item [(2$S$)] 
As in (S1), we want to construct a non-chordal cycle of $G$.
We write $k=\lfloor \frac{k}{2} \rfloor + \lceil \frac{k}{2} \rceil$ and $\lfloor \frac{k}{2} \rfloor=qr+t$ with $0 \leq t \leq r-1$. Now we write $\lceil \frac{k}{2} \rceil = qr+s$, where
\[
s=\begin{cases} t \ \ &\mbox{if } k \ \mbox{even} \\ t+1 \ \ &\mbox{if } k \ \mbox{odd} \end{cases}
\]
Then we take the cycle on vertices
\begin{equation}\label{cycle}
\Big\{0, ra, 2ra,  \ldots , qra, \Big\lfloor \frac{k}{2} \Big\rfloor a, \Big\lfloor \frac{k}{2}a \Big\rfloor + ra,
\Big\lfloor \frac{k}{2}a \Big\rfloor, \ldots \Big\lfloor \frac{k}{2} \Big\rfloor + qra \Big\}.
\end{equation}
Since $r<\lfloor \frac{k}{2} \rfloor$, then $q\geq 1$ and in the case $q=1$, $s > 0$. That is, the cycle on vertices $\eqref{cycle}$ has length at least 4 and it is not chordal because by construction any pair of non-adjacent vertices in the cycle has labelling distance in $\{(r+1)a,\ldots, \lfloor \frac{k}{2} \rfloor a\}$.\\
In any case $G$ is not chordal and the assertion follows.
\end{itemize}
\ 
\end{proof}
An immediate consequence of the previous Lemma is 
\begin{Corollary}\label{gcd}
Let $G=C_n(S)$ be a circulant graph.
If there exists $a\in S$ with $k=\ord(a)\geq 4$ such that $\gcd(a,n) \notin S$, then $G$ is not chordal.
\end{Corollary} 

\begin{Lemma}\label{notch}
Let $G=C_n(S)$ be a circulant graph. 
If $a_{1},\ldots,a_{r} \in S$ and $\gcd(a_{1}, \ldots,a_{r}) \notin S$ then $G$ is not chordal.
\end{Lemma}
\begin{proof}
We proceed by induction on $r$.\\
Let $r=2$ and let $a_{1},a_{2}\in S$ be such that $c=\gcd(a_1,a_2) \notin S$. We consider
\[
a=\gcd(a_1,n), \ b=\gcd(a_2,n), \ d=\gcd(a,b).
\]
From Corollary \ref{gcd}, we have that if one between $a,b$ does not belong to $S$, then $G$ is not chordal. Hence $a,b \in S$.
We have that $d$ divides $c$ and we distinguish two cases. If $d \in S$, since $c=td \notin S$ for some $t$, then by Lemma \ref{orda} $G$ is not chordal. Therefore, from now on we suppose $d \notin S$. Since $a$ and $b$ divide $n$, then $\lcm(a,b)=\frac{ab}{d}$ divides $n$. We want to find a non-chordal cycle of $G$ having length 4.
Let $ra+sb=d \ \pmod n$ be a B\'{e}zout identity of $a$ and $b$. From Lemma \ref{orda}, if one between $ra$ and $sb$ is not in $S$, then $G$ is not chordal. Hence, let us assume $ra,sb\in S$. Now we consider the cycle
\[
\{0,ra,ra+sb=d,sb\}
\]
Since $d \notin S$, then the edge $\{0,d\} \notin E(G)$. We distinguish two cases about $ra-sb$. If $ra-sb \notin S$,
then the assertion follows.\\
If $ra-sb \in S$ we set 
\[
kd=\gcd(ra-sb,n ) \Rightarrow k= \gcd\Big(r\Big(\frac{a}{d}\Big)+s \Big( \frac{b}{d} \Big),\frac{n}{d}\Big).
\]
If $kd$ is not in $S$, then from Corollary \ref{gcd} $G$ is not chordal. Hence, we consider $kd\in S$. 
Since $\gcd\Big(\frac{a}{d},\frac{b}{d}\Big) =1$, then $\gcd\Big(k,\frac{a}{d}\Big)=\gcd\Big(k,\frac{b}{d}\Big) =1$, and
\begin{equation}\label{kd}
\gcd\Big(k, \frac{ab}{d^2}\Big)=1 \ \Rightarrow \gcd \Big(kd, \frac{ab}{d} \Big)=d.
\end{equation}
Hence $\lcm \Big(kd, \frac{ab}{d}\Big)= k \frac{ab}{d}$ divides $n$. We distinguish two cases.
If $k=1$, we obtain the contradiction $d \in S$, arising from the assumption $ra-sb \in S$. If $k\neq 1$, $k$ is a new proper divisor of $n$. We set $a'=kd$ and $b'=\frac{ab}{d}$, we apply the steps above and we find a $k'$ so that $k' \frac{a'b'}{d}$ divides $n$, and so on. By applying the steps above to $a'$ and $b'$ a finite number of times, we could either find a $k'$ equal to 1 or we could get new proper divisors of $n$, that are finite in number. We want to study the case $n=\frac{a'b'}{d}$. 
Let 
\[
va' + z b' =d
\] 
be a B\'{e}zout identity, we assume $va' - z b' \in S$, and we set
\[
hd=\gcd\Big(va' + z b',n\Big).
\]
We have that $h \frac{a'b'}{d}=hn$ divides $n$, that is $hn=n$ and $h=1$. It implies $d \in S$, that is a contradiction  arising from the assumption $va' - z b' \in S$. Hence $va' - z b' \notin S$ and $\{0,va',d,z b'\}$ is a non-chordal cycle of $G$.
It ends the induction basis.
For the inductive step, we suppose the statement true for $r-1$ and we prove it for $r$.
We have to prove that if $\gcd(a_{1},\ldots,a_{r}) \notin S$ then $G$ is not chordal. By inductive hypothesis if $\gcd(a_{1},\ldots,a_{r-1}) \notin S$ then $G$ will be not chordal. Hence we assume $b=\gcd(a_{1},\ldots,a_{r-1})\in S$. By applying the inductive basis to $a_r$ and $b$,we obtain that $G$ is not chordal. 
\end{proof}

Now we are able to complete the proof of Theorem \ref{chord}.

\begin{proof}[Proof of Theorem \ref{chord}.$\Rightarrow )$.]
Under the hypothesis that $G$ is chordal, we also assume that $G$ is connected and we prove that $d=1$, that is $G=K_n$. By contradiction assume that the graph is not complete, namely $G=C_n(a_1,\ldots, a_s)$ with $s<\lfloor \frac{n}{2} \rfloor$. From Lemma \ref{comp}, $G$ is connected if and only if $\gcd(a_{1},\ldots,a_{s},n)=1$.
%We first observe that if there exists $a_{i} \in S$ with $\gcd(a_{i},n)=1$, then $G$ is not chordal from Lemma \ref{orda}. \\
Let $b= \gcd(a_{1},\ldots,a_{s})$. \\
If $b\notin S$, then from Lemma \ref{notch} $G$ is not chordal. If $b\in S$, we have $1=\gcd(n,a_{1},\ldots,a_{s})=\gcd(n,\gcd(a_{1},\ldots,a_{s}))=\gcd(n,b)$. If $1 \notin S$, then from Lemma \ref{notch}, $G$ is not chordal. Then $1 \in S$ and from Lemma \ref{orda} the graph $G$ is not chordal, that is a contradiction.
If $G$ is not connected, then it has $a=\gcd(n,S)$ distinct components, each of $m=\ord(a)$ vertices. By Lemma \ref{orda}, $S=\{a,2a,\ldots, \lfloor \frac{m}{2} \rfloor a \}$ and each component is the complete graph $K_{m}$.
\end{proof}

\begin{Example}
Here we present three examples of non-chordal circulant graphs $C_n(S)$.
\begin{enumerate}
\item Take $n=15$ and $S=\{2,3,4,7\}$. If we take $a=2$, then $\ord (a)=15$ and $2a=4$, $3a=6$, $n-4a=7$, and $n-6a=3$. Hence, we are in case (1$S$) of Lemma \ref{orda} with $S=\{a,2a,4a,6a \}$.
We observe that the cycle on vertices
\[
\{0,a,3a,4a\}=\{0,2,6,8\}
\]
is not chordal because $6 \notin S$.
\begin{figure}[h]
\centering
\includegraphics[scale=0.50]{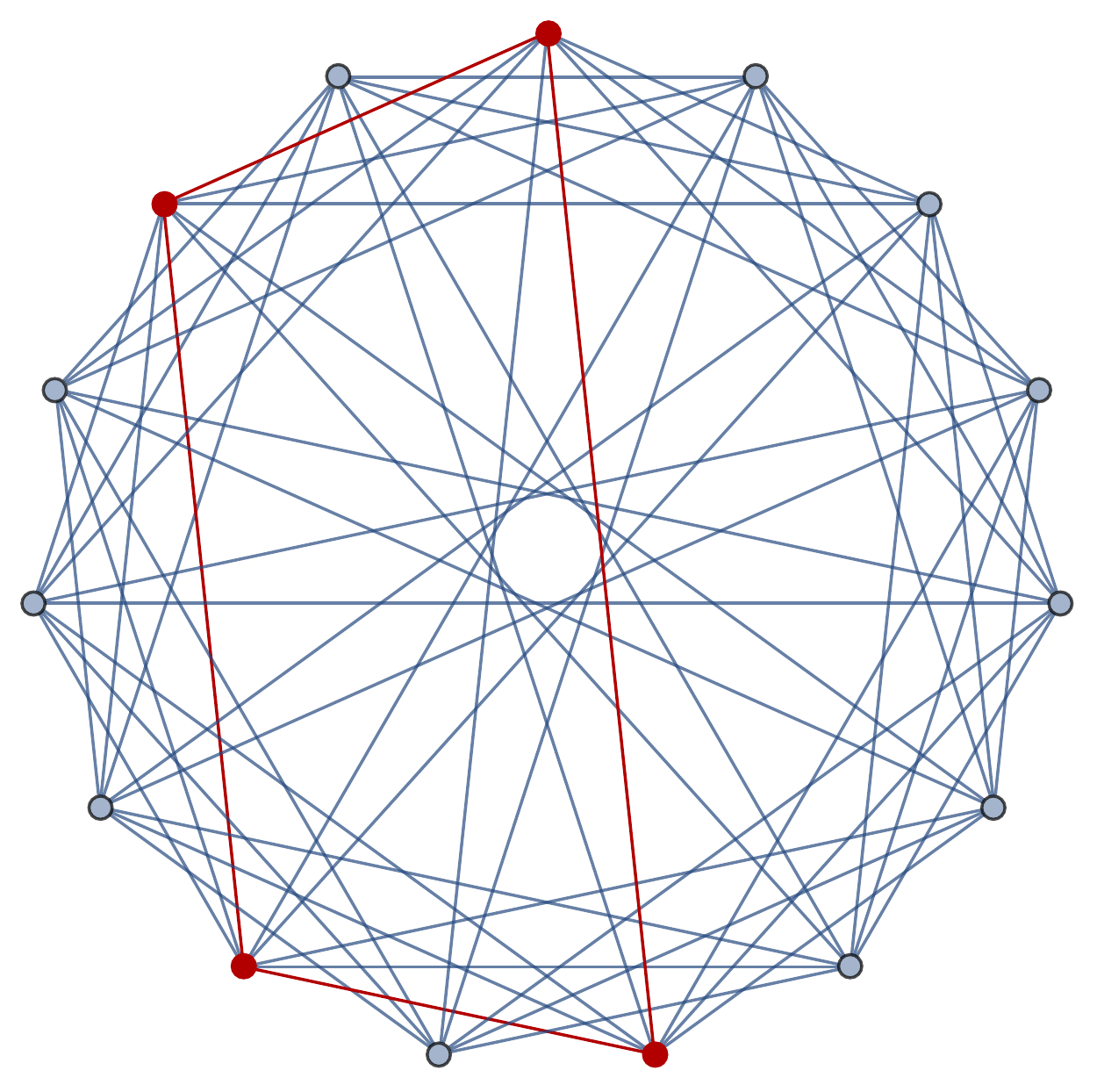}\caption{$C_{15}(2,3,4,7)$}\label{fig5}
\end{figure}
\begin{figure}[h]
\centering
\includegraphics[scale=0.50]{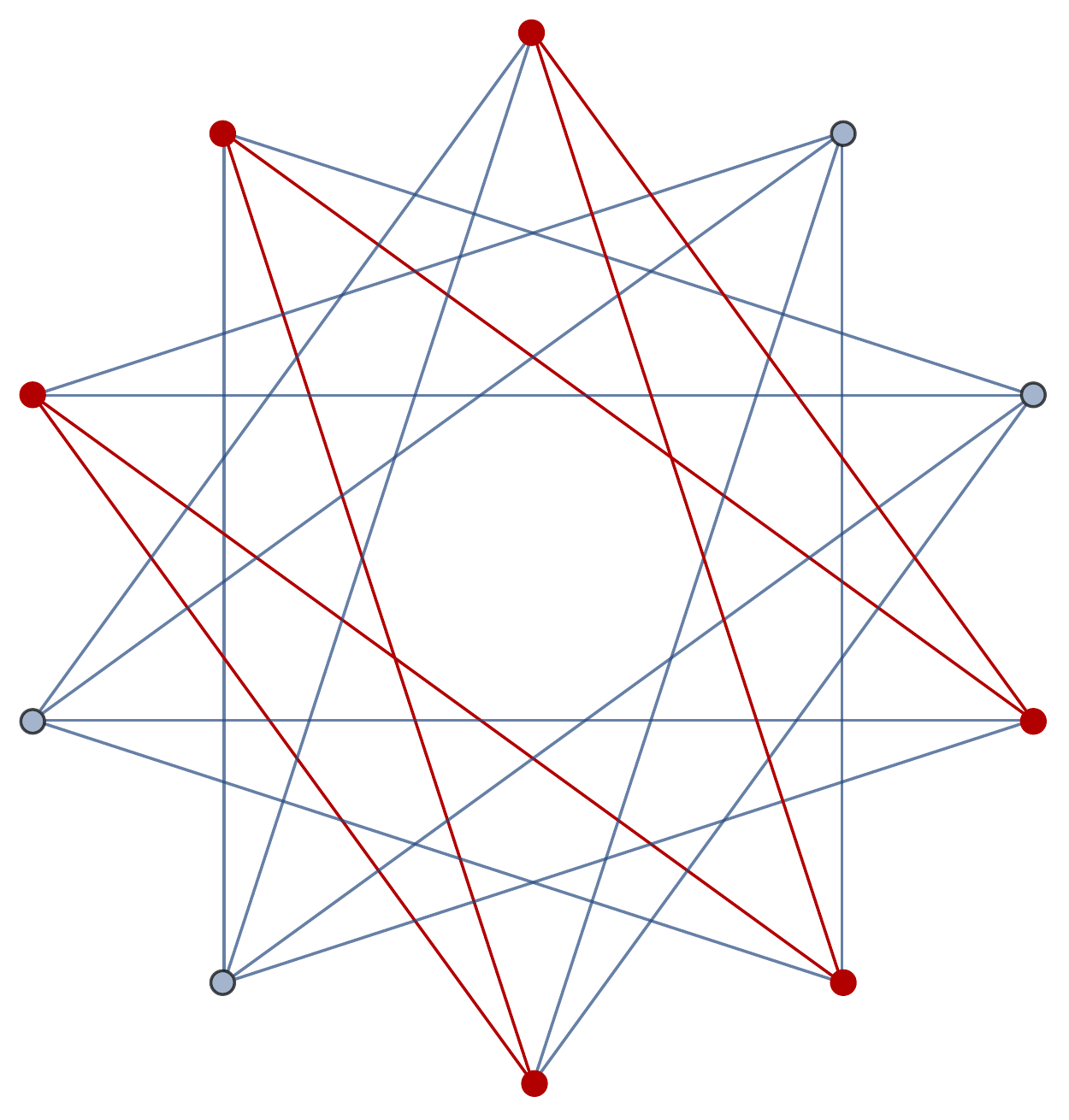}\caption{$C_{10}(3,4)$}\label{fig6} 
\end{figure}
\begin{figure}[h]
\centering
\includegraphics[scale=0.30]{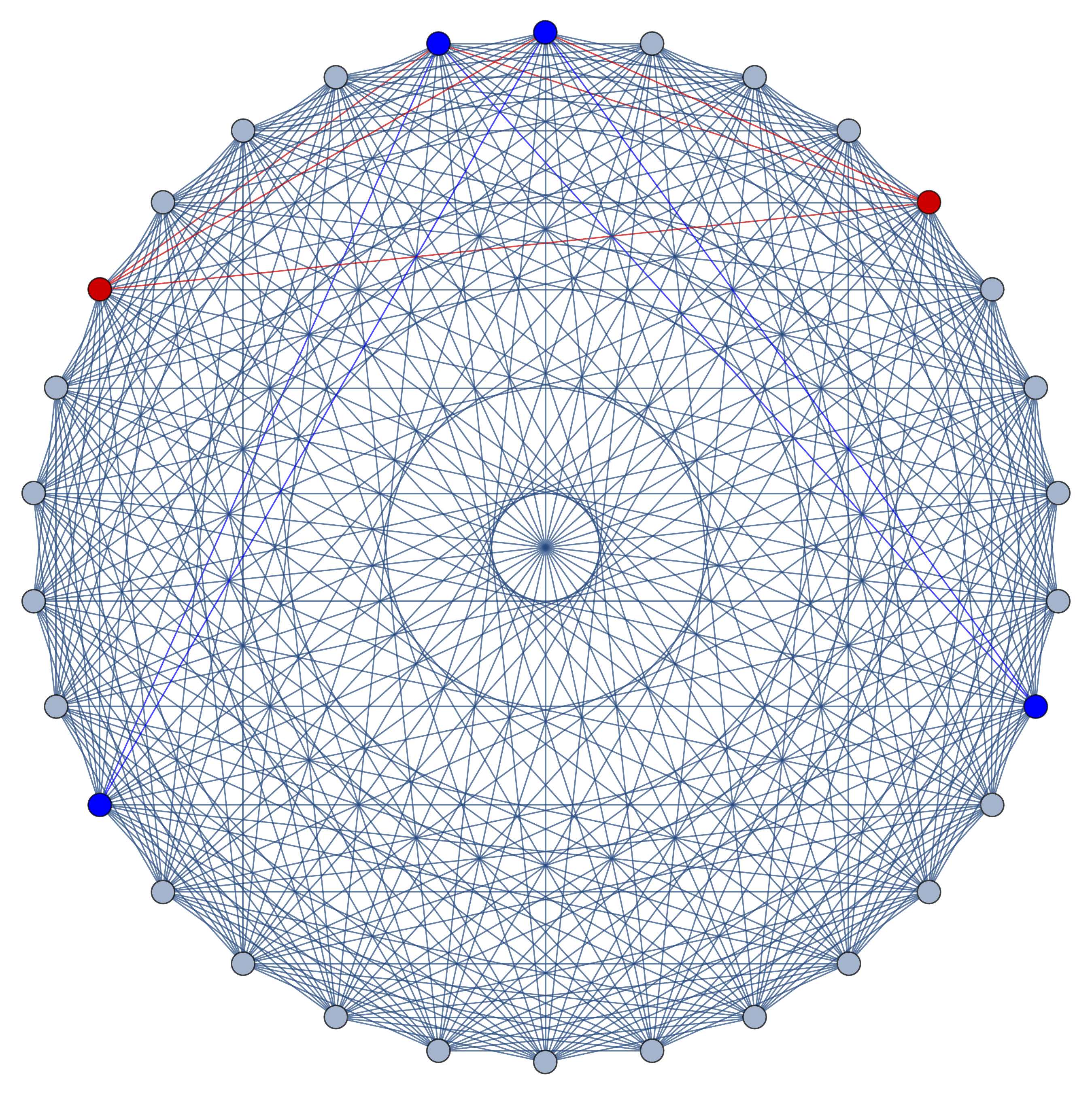}\caption{$C_{30}(2,3,4,5,6,8,9,10,12,14,15)$}\label{fig7}
\end{figure}
\item Take $n=10$, $S=\{3,4 \}$ and $a=3$. We have $\ord(a)=10$. Moreover $n-2a=4$, hence this is the case (2$S$) of Lemma \ref{orda} with $S=\{a,2a \}$. We have $\lfloor \frac{n}{2} \rfloor=\lceil \frac{n}{2} \rceil=5$, and 
\[
5=qr+t=2 \cdot 2+1.
\]
Hence, we take the cycle on vertices
\[
\{0,2a,4a,5a, 7a ,9a\}= \{0, 6, 2, 5, 1, 7\}
\]
it is not chordal because $1,2$ and $5$ do not belong to $S$.
\item We take $n=30$ and $S=\{2,3,4,5,6,8,9,10,12,14,15 \}$. We observe that $\gcd(5,2)=1\notin S$, hence we are in the case of Lemma \ref{notch} with $a_1=a=5$ and $a_2=b=2$. We observe that $\ord(a)=6,ord(b)=15$ and $2a=10, 3a=15,b,2b,\ldots,7b \in S$. We take a Bez\'{o}ut identity of $a$ and $b$
\[
1=ra+sb=5\cdot 1 -2 \cdot 2.
\]
We take the cycle on vertices $\{0,5, 1, -4\}$.
The quantity $ra-sb=5+4=9$ belongs to $S$ and $k=\gcd(9,30)=3$, while $\gcd(k,ab)=\gcd(3,10)=1$ and $n=abk=30$. Hence we write
\[1=vab+sk=10 - 3 \cdot 3\]
and we take the cycle on vertices $\{ 0, 10,1,-9\}$.
The quantity $10+9=19$ does not belong to $S$, hence the cycle above is not chordal.
\end{enumerate}
In Figures \ref{fig5},\ref{fig6},\ref{fig7} we plot the graphs of the examples, highlighting the non-chordal cycles.
\end{Example}

\section{Induced Matching Number of Circulant Graphs}\label{sec:two}
In this section we compute the induced matching number for any circulant graph $C_n(S).$ Then we plot a table representing the behaviour of $\reg R/I(G)$ with respect to the lower bound described in Theorem \ref{kat}, when $G$ is the $d$-th power of the cycle, namely $G=C_{n}(1,2,\ldots, d)$. For the computation we used \texttt{Macaulay2}.

 \begin{Definition}\label{adj}
Let $G$ be a graph with edge set $E(G)$. We say that two edges $e,e'$ are \emph{adjacent} if $e \cap e'={v}$ and $ v \in V(G)$. We say that $e,e'$ are \emph{$2$-adjacent} if there exist $v \in e$ and $u \in e'$ such that $\{u,v \} \in E(G)$.
\end{Definition}

From Definition \ref{adj}, an induced matching of $G$ is a subset of $E(G)$ where the edges are not pairwise adjacent or 2-adjacent. Then we have the following

\begin{Theorem}\label{indma}
Let $G=C_{n}(S)$ be a connected circulant graph, let $s=|S|$ and let $r=\min S$. Then $\nu (G)=\lfloor \frac{|E(G)|}{t} \rfloor$
where \[t=\begin{cases} s^2+(|A|+1)s   &\mbox{if } \frac{n}{2} \notin S \\  s^2+(|A|+1)s-2  &\mbox{if } \frac{n}{2} \in S, \ \end{cases}\] with
\[
A=\Big\{r+a \ : \ a \in S \ \mbox{ and } \  r+a\in V(G)\setminus S \Big\}.
\]
If $G$ has $d=gcd(n,S)$ components, then $\nu (G)= d \cdot \nu(C_{n/d}(S'))$, where $S'=\{s/d \ : \ s \in S \}$. 
\end{Theorem}
\begin{proof}
To explain the idea of proof, we first study the simple case. It is well known that the cycle $C_{n}$ has $\nu$ equals to $\lfloor \frac{n}{3}\rfloor$. It happens because, by fixing an orientation, to get the induced matching we partition the $n$ edges of the cycle in sets of $3$ adjacent edges, one in the matching, one not in the induced matching because adjacent to the first one and another one not in the induced matching because 2-adjacent to the first one. We observe that $A=\{2\}$ and the formula $t=1+(1+1)\cdot 1=3$ holds in this case.
Hence the cardinality of the maximum induced matching is equal to the number of the sets above,
\[
\nu(G)=\Big\lfloor \frac{n}{3} \Big\rfloor.
\]
The example shows that $\nu (G)$ corresponds to the number of sets consisting in one edge in the matching and the adjacent or 2-adjacent edges to that one. So we have only to count the edges.\\
We assume that $s=|S|$, $r=\min S$ and $S=\{a_{0}=r,a_{1},\ldots ,a_{s-1}\}$, we assume that the edge $e=\{0,r\}$ is in the induced matching, and let $E'$ be the set containing $e$ and the edges adjacent or 2-adjacent to $e$. The edges adjacent to $e$ are $\{0,a_{i}\}$ $i=1,\ldots,s-1$ and $\{r,b_{i}=r+a_{i}\}$ for $i=0,\ldots,s-1$. The above edges are all distinct. The edges $2$-adjacent to $e$ are $\{a_{j},a_{j}+a_{i}\}$ for $j \in \{1,\ldots, s-1\},i\in \{0,\ldots,s-1\}$ and $\{b_{j},b_{j}+a_{i}\}$ for $i,j\in\{0,\ldots,s-1\}$. The edges above may not be all distinct. In fact, it can happen that some $b_{j}$ coincides with some $a_{k}$, in that case $\{b_{j},b_{j}+a_{i}\}=\{a_{k},a_{k}+a_{i}\}$ for any $i \in \{0,\ldots,s-1\}$. Then, we only consider $\{b_{j},b_{j}+a_{i}\}$ for $i \in \{0,\ldots,s-1\}$ when $b_{j} \in A$. To sum up, in the set $E'$ we find: 
\begin{itemize}
\item[a)] The $s$ edges $\{0,a_{i}\}$ for $i \in \{0,\ldots,s-1\}$;
\item[b)] The $s^2$ edges $\{a_{j},a_{j}+a_{i}\}$ for $i,j \in \{0,\ldots,s-1\}$;
\item[c)] The $s\cdot |A|$ edges $\{b,b+a_{i}\}$ for $i \in \{0,\ldots,s-1\}$ and $b\in A$.
\end{itemize}
If $a_{s-1}=\frac{n}{2}$, then $b_{s-1}=r+a_{s-1}\in A$ and the edges $\{a_{s-1},a_{s-1}+a_{s-1}=0\}$ of point b) and $\{b_{s-1},b_{s-1}+a_{s-1}=r \}$  of point c) are already counted. The assertion follows.

%We show how to get the formula, starting by the complete graph. From vertex 0 there are $s$ edges in the first half of the graph. We take as edge in the matching $\{0,1\}$, hence we have to exclude those edges adjacent or 2-adjacent  to this one and the adjacent edges to the remaining $s-1$. We have $s^2$ adjacent edges (from any ending vertex of the first $s$ edges consider other $s$ edges). Considering now the $2$-adjacent to $\{0,1\}$, since the graph is complete, we have already took off the $2$-adjacent edges $\{i,i+s\}$ with $i,s \in S$, hence $A=\{\lfloor\frac{n}{2}\rfloor +1 \}$. Finally we have 
%\begin{enumerate}
%\item $s$ starting edges;
%\item $s^2$ or $s^2 -1$ (if $\frac{n}{2} \in S$ ) edges adjacent to the first ones;
%\item $s$ or $s-1$ edges starting from the vertex $\lfloor\frac{n}{2}\rfloor\ +1$.
%\end{enumerate}
%Therefore we get $t=s^2+2s$ or $s^2+2s-2$.
%In general we take a connected graph $C_{n}(S)$, we take the first $s^2+s$ (or $s^2+s-2$) edges of before. Now if we focus on 
%$\{0,r\}$, from any $\{r,r+a\} \in A$ there are more $s$ 2-adjacent edges to $\{0,r\}$. 

%Furthermore, the edges $\{b,a \}$ with $a,b \in S$ and $b-a \notin S$ will be $2$-adjacent to $\{0,r\}$. Hence the assertion follows.

For the case disconnected, let $d=gcd(n,S)$ be the number of disjoint connected components of the graph G. Since the components are disjoint, it turns out that $\nu (G)$ is $d$ times the induced matching number of one component. That component is $C_{n/d}(S')$
where $S'=\{s/d \ : \ s \in S \}$, hence the assertion follows.
\end{proof}

The formula in Theorem \ref{indma} can be written in a compact way when $G$ is the $d$-th power of a cycle. We set $C_{n}^d=C_{n}(\{1,2,\ldots,d\})$.
\begin{Corollary}\label{indpow}
Let $C_{n}^d$ be the $d$-th power of a cycle and $d < \lfloor \frac{n}{2} \rfloor $. Then
\[
\nu (G)= \Big\lfloor \frac{n}{d+2} \Big\rfloor.
\]
\end{Corollary}
\begin{proof}
We want to apply Proposition \ref{indma}, with $s=d$ and $|E(G)|=nd$. We have $r=1$ and $A=\{d+1\}$. Hence it follows that 
$t=d^2+d+d \cdot 1=d^2+2d=d(d+2)$, that is
\[
\nu (G)= \Big\lfloor \frac{nd}{d(d+2)} \Big\rfloor= \Big\lfloor \frac{n}{d+2} \Big\rfloor.
\]
\end{proof}

In Table \ref{Tab}, we compare the values of $\reg R/I(C_n^d)$ for $n \leq 15$ and $1 \leq d \leq \lfloor\frac{n}{2} \rfloor $. We highlight that the regularity of $R/I(G)$ is strictly greater than $\nu(G)$ in two different cases:
\begin{itemize}
\item[(1)] when $G=C_n$ and $n \equiv 2 \pmod 3$.
\item[(2)] when $G=C_{n}^{\lfloor\frac{n}{2} \rfloor-1}$ and $n$ is odd.
\end{itemize}
The two anomalous cases were expected: in case (1), we know from Theorem \ref{cyc} that $\reg R/I(G)=\nu +1$; in case (2), $\nu(G)=1$ while $\bar{G}=C_n(\lfloor\frac{n}{2} \rfloor)$ that is a cycle and hence it is not chordal; hence from Theorem \ref{fr} we know that $\reg R/I(G)=2$. 

In general, it seems that apart from cases (1) and (2), the Castelnuovo-Mumford regularity of the $d$-th power of a cycle grips the bound of $\nu(G)$.

\begin{table}[h]
\centering
\begin{small}
\begin{tabular}{| l | c | c | | l | c | c | }
\hline
$\ \ \ \ \ \ \ \ \ \ G$ &$\nu (G)$ &$\reg R/I(G)$ &$\ \ \ \ \ \ \ \ \ \ \ \ G$ &$\nu (G)$ &$\reg R/I(G)$\\
\hline
$C_6(\{1\})$ &2    & 2                    &$C_{12}(\{1,2,3\})$              &2    &2 \\
$C_6(\{1,2\})$                 &1    &1  &$C_{12}(\{1,2,3,4\})$            &2    &2 \\
$C_7(\{1\})$                    &2    &2  &$C_{12}(\{1,2,3,4,5\})$          &1    &1 \\
$C_7(\{1,2\})$                 &1    &2  &$C_{13}(\{1\})$                   &4    &4 \\
$C_8(\{1\})$                    &2    &3  &$C_{13}(\{1,2\})$                &3    &3 \\
$C_8(\{1,2\})$                 &2    &2  &$C_{13}(\{1,2,3\})$              &2    &2 \\
$C_8(\{1,2,3\})$               &1    &1  &$C_{13}(\{1,2,3,4\})$            &2    &2 \\
$C_9(\{1\})$                    &3    &3  &$C_{13}(\{1,2,3,4,5\})$          &1    &2 \\
$C_9(\{1,2\})$                 &2    &2  &$C_{14}(\{1\})$                   &4    &5 \\ 
$C_9(\{1,2,3\})$               &1    &2  &$C_{14}(\{1,2\})$                &3    &3 \\
$C_{10}(\{1\})$                   &3    &3  &$C_{14}(\{1,2,3\})$              &2    &2 \\
$C_{10}(\{1,2\})$                &2    &2  &$C_{14}(\{1,2,3,4\})$            &2    &2 \\
$C_{10}(\{1,2,3\})$              &2    &2  &$C_{14}(\{1,2,3,4,5\})$          &2    &2 \\
$C_{10}(\{1,2,3,4\})$            &1    &1  &$C_{14}(\{1,2,3,4,5,6\})$        &1    &1 \\
$C_{11}(\{1\})$                   &3    &4  &$C_{15}(\{1\})$                   &5    &5 \\
$C_{11}(\{1,2\})$                &2    &2  &$C_{15}(\{1,2\})$                &3    &3 \\
$C_{11}(\{1,2,3\})$              &2    &2  &$C_{15}(\{1,2,3\})$              &3    &3 \\
$C_{11}(\{1,2,3,4\})$            &1    &2  &$C_{15}(\{1,2,3,4\})$            &2    &2 \\
$C_{12}(\{1\})$                   &4    &4  &$C_{15}(\{1,2,3,4,5\})$          &2    &2 \\
$C_{12}(\{1,2\})$                &3    &3  &$C_{15}(\{1,2,3,4,5,6\})$        &1    &2  \\
\hline
\end{tabular}\caption{The behavior of $\reg R/I(G)$ with respect to $\nu(G)$ for $G=C_n^d$.}\label{Tab}
\end{small}
\end{table}
\ \\
\textbf{Acknowledgements:} I would like to thank Giancarlo Rinaldo for going through the manuscript and making some useful suggestions.

\end{document}